\documentclass[11pt]{article}

\usepackage{amssymb,amsmath}

\usepackage{amsmath}
\usepackage{amssymb}
\usepackage{amsthm}
\usepackage{url}
\usepackage{graphicx} 
\usepackage{rotating}

\newcommand\x{{\bf x}}

\newcommand\z{{\bf z}}
\newcommand\bfa{{\bf a}}
\newcommand\bfb{{\bf b}}
\newcommand\bfe{{\bf e}}
\newcommand\bff{{\bf f}}
\newcommand\bfv{{\bf v}}
\newcommand\zero{{\bf 0}}
\newcommand\cc{{\mathbb C}}

\newcommand\nn{{\mathbb N}}

\newtheorem{theorem}{Theorem}[section]
\newtheorem{lemma}[theorem]{Lemma}
\newtheorem{proposition}[theorem]{Proposition}

\theoremstyle{definition}
\newtheorem{definition}[theorem]{Definition}

\begin{document}

\title{Polyhedral Methods for Space Curves Exploiting Symmetry 
Applied to the Cyclic $n$-roots Problem\thanks{This
material is based upon work supported by the National Science Foundation
under Grant No.\ 0713018 and 
Grant No.\ 1115777.} }

\author{Danko Adrovic and Jan Verschelde\\
Department of Mathematics, Statistics, and Computer Science \\
University of Illinois at Chicago \\
851 South Morgan (M/C 249) \\
Chicago, IL 60607-7045, USA \\
  \texttt{adrovic@math.uic.edu, jan@math.uic.edu}\\
  \texttt{\url{www.math.uic.edu/~adrovic}, \url{www.math.uic.edu/~jan}}}

\date{2 June 2013}

\maketitle

\begin{abstract}
We present a polyhedral algorithm to manipulate positive dimensional
solution sets. Using facet normals to Newton polytopes as pretropisms,
we focus on the first two terms of a Puiseux series expansion. 
The leading powers of the series are computed via the tropical
prevariety.
This polyhedral algorithm is well suited for exploitation of symmetry,
when it arises in systems of polynomials. Initial form systems with pretropisms in the same group orbit are solved only once, allowing for a systematic
filtration of redundant data.
Computations with cddlib, Gfan, PHCpack, and Sage
are illustrated on cyclic $n$-roots polynomial systems.

\end{abstract}

\section{Introduction}

We consider a polynomial system $\bff(\x) = \zero$,
$\x = (x_0,x_1,\ldots,x_{n-1})$, $\bff = (f_1,f_2,\ldots,f_N)$,
$f_i \in \cc[\x]$, $i=1,2,\ldots,N$.  Although in many applications
the coefficients of the polynomials are rational numbers, we allow
the input system to have approximate complex numbers as coefficients.
For $N = n$ (as many equations as unknowns),
we expect in general to find only isolated solutions. 
In this paper we focus on cases $N \geq n$ where the coefficients are 
so special that $\bff(\x) = \zero$ has an algebraic set as a solution.

Our approach is based on the following observation:
if the solution set of~$\bff(\x) = \zero$ has a space curve,
then this space curve extends from~$\cc^* = \cc \setminus \{ 0 \}$
to infinity.  In particular, the space curve intersects hyperplanes
at infinity at isolated points.  We start our series development of
the space curve at these isolated points.
Computing series developments for solutions of polynomial systems
is a hybrid symbolic-numeric method, appropriate for inputs which
consist of approximate numbers (the coefficients) and exact data 
(the exponents).

In this paper we will make various significant assumptions.
First we assume that the algebraic sets we consider are reduced,
that is: free of multiplicities.  Moreover, an algebraic set 
of dimension~$d$ is
in general position with respect to the first $d$ coordinate planes.
For example, we assume that a space curve is not contained
in a plane perpendicular to the first coordinate axis.
Thirdly, we assume the algebraic set of dimension~$d$ to intersect 
the first $d$ coordinate planes at regular solutions.  

Our approach consists of two stages.
The computation of the candidates for 
the leading powers of the Puiseux series is followed by
the computation of the leading coefficients and the second
term of the Puiseux series, if the leading term of the series 
does not already entirely satisfy the system.  
Following our assumptions, the second term of the Puiseux series 
indicates the existence of a space curve. 
If the system is invariant to permutation of the variables, then it
suffices to compute only the generators of the solution orbits.
We then develop the Puiseux series only at the generators.
Although our approach is directed at general algebraic sets,
our approach of exploiting symmetry applies also to the computation
of all isolated solutions.
Our main example is one family of polynomial systems,
the cyclic $n$-roots system. 

\noindent {\bf Related Work.}
Our approach is inspired by the constructive proof of the 
fundamental theorem of tropical algebraic geometry in~\cite{JMM08}
(an alternative proof is in~\cite{Pay09}) and related to finiteness
proofs in celestial mechanics~\cite{HM06}, \cite{JH11}.
The initial form systems allow the elimination of variables with
the application of coordinate transformations, an approach
presented in~\cite{HL12} and related to the application of
the Smith normal form in~\cite{GW12}.
The complexity of polyhedral homotopies is studied
in~\cite{JMSW09} and generalized to affine solutions in~\cite{HJS13}.
Generalizations of the Newton-Puiseux theorem~\cite{Pui50},
\cite{Wal50}, can be found in \cite{AIL11}, \cite{BR03}, \cite{Mau80},
\cite{McD02}, \cite{Ron13}, and~\cite{Sev13}.
A symbolic-numeric computation of Puiseux series is described
in~\cite{Pot08}, \cite{PR08}, and~\cite{PR12}.
Algebraic approaches to exploit symmetry
are~\cite{Col97}, \cite{FR09}, \cite{Gat00}, and~\cite{Ste13}.
The cyclic $n$-roots problem is a benchmark for polynomial system solvers,
see e.g: \cite{BF94}, \cite{Col97}, \cite{DKK03}, \cite{Emi94}, \cite{EC95}, 
\cite{Fau99}, \cite{FR09}, \cite{LLT08}, \cite{Ste13},
and relevant to operator algebras~\cite{BH07},
\cite{Haa07}, \cite{Szo11}.
Our results on cyclic 12-roots correspond to~\cite{Sab11}.

\noindent {\bf Our Contributions.}
This paper is a thorough revision of the unpublished preprint~\cite{AV11b},
originating in the dissertation of the first author~\cite{Adr12},
which extended~\cite{AV11} from the plane to space curves.
In~\cite{AV12} we gave a tropical version of Backelin's Lemma
in case $n = m^2$, in this paper we generalize to the case~$n = \ell m^2$.
Our approach improves homotopies to find all isolated solutions.
Exploiting symmetry we compute only the generating
cyclic $n$-roots, more efficiently than the
symmetric polyhedral homotopies of~\cite{VG95}.

\section{Initial Forms, Cyclic $n$-roots, and Backelin's Lemma}

In this section we introduce our approach on the cyclic 4-roots problem.
For this problem we can compute an explicit representation for the 
solution curves.
This explicit representation as monomials in the independent parameters
for positive dimensional solution sets generalizes into the tropical
version of Backelin's Lemma.

\subsection{Newton Polytopes, Initial Forms, and Tropisms}

In this section we first define Newton polytopes, initial forms,
pretropisms, and tropisms.  The sparse structure of a polynomial
system is captured by the sets of exponents and their convex hulls.

\begin{definition} {\rm 
Formally we denote a polynomial $f \in \cc[\x]$ as
\begin{equation}
   f(\x) = \sum_{\bfa \in A} c_\bfa \x^\bfa, \quad c_\bfa \in \cc^*,
   \quad \x^\bfa = x_0^{a_0} x_1^{a_1} \cdots x_{n-1}^{a_{n-1}},
\end{equation}
and we call the set $A$ of exponents {\em the support of~$f$}.
The convex hull of~$A$ is {\em the Newton polytope of~$f$}.
The tuple of supports 
${\bf A} = (A_1,A_2,\ldots,A_N)$ span the Newton polytopes
${\bf P} = (P_1,P_2,\ldots,P_N)$ of the polynomials
$\bff = (f_1,f_2,\ldots,f_N)$ of the system~$\bff(\x) = \zero$. }
\end{definition}

The development of a series starts at a solution of an initial form 
of the system~$\bff(\x) = \zero$,
with supports that span faces of the Newton polytopes of~$\bff$.

\begin{definition} {\rm Let $\bfv \not= \zero$, denote 
$\langle \bfa , \bfv \rangle = a_0 v_0 + a_1 v_1 + \cdots + a_{n-1} v_{n-1}$,
and let $f$ be a polynomial supported on~$A$.
Then, {\em the initial form of $f$
in the direction of~$\bfv$} is

\begin{equation}
{\rm in}_\bfv(f)
  = \sum_{\begin{array}{c}
           \bfa \in {\rm in}_\bfv(A)
        \end{array}} c_\bfa \x^\bfa,
  \quad \mbox{where} \quad
  {\rm in}_\bfv(A)
   = \{ \ \bfa \in A  \ | \ \langle \bfa , \bfv \rangle
    = \min_{\bfb \in A} \langle \bfb , \bfv \rangle  \ \}.
\end{equation}
The initial form of a system $\bff(\x) = \zero$ with polynomials 
in $\bff = (f_1,f_2,\ldots,f_N)$ 
in the direction of~$\bfv$ is denoted by 
${\rm in}_\bfv(\bff) = ({\rm in}_\bfv(f_1),{\rm in}_\bfv(f_2), 
\ldots, {\rm in}_\bfv(f_N))$.
If the number of monomials with nonzero coefficient in each
${\rm in}_\bfv(f_k)$, for all $k=1,2,\ldots,N$, is at least two,
then $\bfv$ is a {\em pretropism}.  }
\end{definition}
The notation ${\rm in}_\bfv(f)$ follows~\cite{Stu96}, where
$\bfv$ represents a weight vector to order monomials.
The polynomial ${\rm in}_\bfv(f)$ is homogeneous with respect to~$\bfv$.
Therefore, in solutions of ${\rm in}_\bfv(f)(\x) = \zero$
we can set $x_0$ to the free parameter~$t$.
In~\cite{Bru00} and~\cite{Kaz99}, initial form systems are called
truncated systems. 

Faces of Newton polytopes $P$ spanned by two points are edges
and all vectors~$\bfv$ that lead to the same ${\rm in}_\bfv(P)$
(the convex hull of ${\rm in}_\bfv(A)$) define a polyhedral cone
(see e.g.~\cite{Zie95} for an introduction to polytopes).

\begin{definition} {\rm
Given a tuple of Newton polytopes $\bf P$ of a system $\bff(\x) = \zero$,
{\em the tropical prevariety} of~$\bff$ is the common refinement of the
normal cones to the edges of the Newton polytopes in~$\bf P$.  }
\end{definition}

Our definition of a tropical prevariety is based on the algorithmic
characterization in~\cite[Algorithm~2]{BJSST07},
originating in~\cite{RST05}.  Consider for example the special
case of two polytopes $P_1$ and $P_2$ and take the intersection 
of two cones, normal to two edges of the two polytopes.  
If the intersection is not empty, then the intersection contains 
a vector~$\bfv$ that defines a tuple of two edges
$({\rm in}_\bfv(P_1), {\rm in}_\bfv(P_2))$.

\begin{definition} {\rm 
For space curves, the special role of $x_0$ is reflected
in the normal form of the Puiseux series:
\begin{equation}
   \left\{
      \begin{array}{l}
         x_0 = t^{v_0} \\
         x_i = t^{v_i} ( y_i + z_i t^{w_i}(1+O(t)), \quad i=1,2,\ldots,n-1,
      \end{array}
   \right.
\end{equation}
where the leading powers $\bfv = (v_0,v_1,\ldots,v_{n-1})$
define a {\em tropism}.  
}
\end{definition}
In the definition above, it is important to observe that the tropism~$\bfv$
defines as well the initial form system ${\rm in}_\bfv(\bff)(\x) = \zero$
that has as solution the initial coefficients of the Puiseux series.

Every tropism is a pretropism, but not every pretropism is a tropism, 
because pretropisms depend only on the Newton polytopes.
For a $d$-dimensional algebraic set,
a $d$-dimensional polyhedral cone of tropisms
defines the exponents of Puiseux series
depending on $d$ free parameters.

\subsection{The Cyclic $n$-roots Problem}

For $n=3$, the cyclic $n$-roots system originates naturally from
the elementary symmetric functions in the roots of a cubic polynomial.
For $n=4$, the system is

\begin{equation}
\bff(\x) = \left\{
  \begin{array}{c}
     x_0 + x_1 + x_2 + x_3 = 0 \\
     x_0 x_1 + x_1 x_2 + x_2 x_3 + x_3 x_0 = 0 \\
     x_0 x_1 x_2 + x_1 x_2 x_3 + x_2 x_3 x_0 + x_3 x_0 x_1 = 0 \\
     x_0 x_1 x_2 x_3 - 1 = 0. \\
  \end{array}
\right.
\end{equation}
The permutation group which leaves the equations invariant is generated by
$(x_0,x_1,x_2,x_3) \rightarrow (x_1,x_2,x_3,x_0)$ and
$(x_0,x_1,x_2,x_3) \rightarrow (x_3,x_2,x_1,x_0)$.
In addition, the system is equi-invariant with respect to the action
$(x_0,x_1,x_2,x_3) \rightarrow (x_0^{-1},x_1^{-1},x_2^{-1},x_3^{-1})$.

With $\bfv = (+1,-1,+1,-1)$, there is
a unimodular coordinate transformation~$M$, denoted by $\x = \z^M$:

\begin{equation}
{\rm in}_\bfv(\bff)(\x) = \left\{
  \begin{array}{c}
     x_1 + x_3 = 0 \\
     x_0 x_1 + x_1 x_2 + x_2 x_3 + x_3 x_0 = 0 \\
     x_1 x_2 x_3 + x_3 x_0 x_1 = 0 \\
     x_0 x_1 x_2 x_3 - 1 = 0 \\
  \end{array}
\right.
\quad
\quad
  \x = \z^M:
\left\{
   \begin{array}{l}
      x_0 = z_0^{+1} \\
      x_1 = z_0^{-1} z_1 \\
      x_2 = z_0^{+1} z_2 \\
      x_3 = z_0^{-1} z_3 \\
   \end{array}
\right.
\end{equation}
with
\begin{equation}
  M =
  \left[
     \begin{array}{cccc}
        +1 & -1 & +1 & -1 \\
         0 &  1 &  0 &  0 \\
         0 &  0 &  1 &  0 \\
         0 &  0 &  0 &  1 \\
     \end{array}
  \right].
\end{equation}

The system ${\rm in}_\bfv(\bff)(\z) = \zero$ has two solutions.
These two solutions are the leading coefficients in the
Puiseux series.  In this case, the leading term of the series
vanishes entirely at the system so we write two solution curves as
$\left( t, -t^{-1}, -t, t^{-1} \right)$ and
$\left( t, t^{-1}, -t, -t^{-1} \right)$.
To compute the degree of the two solution curves, we take a random
hyperplane in~$\cc^4$: $c_0 x_0 + c_1 x_1 + c_2 x_2 + c_3 x_3 + c_5 = 0$,
$c_i \in \cc^*$.  Then the number of points on the curve and on the
random hyperplane equals the degree of the curve.
Substituting the representations we obtained for the curves into
the random hyperplanes gives a quadratic polynomial in~$t$ (after
clearing the denominator~$t^{-1})$, so there are two quadric curves
of cyclic 4-roots.

\subsection{A Tropical Version of Backelin's Lemma}

In~\cite{AV12}, we gave an explicit representation for the solution sets
of cyclic $n$-roots, in case $n = m^2$, for any natural number~$m \geq 2$.
Below we state Backelin's Lemma~\cite{Bac89},
in its tropical form.

\begin{lemma}[Tropical Version of Backelin's Lemma] 
For $n = m^2 \ell$, where $\ell \in \nn \setminus \{ 0 \}$ and
$\ell$ is no multiple of~$k^2$, for $k \geq 2$,
there is an $(m-1)$-dimensional set
of cyclic $n$-roots, represented exactly as
\begin{equation} \label{gensol}
   \begin{array}{rcl}
       x_{km+0} & = & u^{k} t_{0} \\
       x_{km+1} & = & u^{k} t_{0} t_{1} \\
       x_{km+2} & = & u^{k} t_{0} t_{1} t_{2} \\
                & \vdots & \\
       x_{km+m-2} & = & u^{k} t_{0} t_{1} t_{2} \cdots t_{m-2} \\
       x_{km+m-1} & = & \gamma u^{k} t_{0}^{-m+1} t_{1}^{-m+2} 
                             \cdots t_{m-3}^{-2} t_{m-2}^{-1}
   \end{array}
\end{equation}
for $k=0,1,2,\ldots, m-1$, free parameters $t_0,t_1,\ldots,t_{m-2}$,
constants $u = e^{\frac{i 2 \pi}{m\ell}}$, 
$\gamma = e^{\frac{i\pi \beta}{m \ell}}$, 
with $\beta = (\alpha \mod 2)$, and 
$\alpha = m(m\ell-1)$.
\end{lemma}

\begin{proof}
By performing the change of variables 
$y_{0}=t_{0}$, $y_{1}=t_{0}t_{1}$, $y_{2}=t_{0}t_{1}t_{2}$, 
$\ldots$, $y_{m-2}=t_{0}t_{1}t_{2} \cdots t_{m-2}$,
$y_{m-1} = \gamma t_{0}^{-m+1} t_{1}^{-m+2} \cdots t_{m-3}^{-2} t_{m-2}^{-1}$,
the solution (\ref{gensol}) can be rewritten as

\begin{equation} \label{gensol-rewritten}
  x_{km+j} = u^{k}y_{j}, \quad j = 0,1,\ldots,m-1.
\end{equation} 
The solution (\ref{gensol-rewritten}) satisfies 
the cyclic $n$-roots system by plain substitution 
as in the proof of \cite[Lemma~1.1]{Fau01},
whenever the last equation $x_{0}x_{1}x_{2}\cdots x_{n-1}-1=0$ 
of the cyclic $n$-roots problem can also be satisfied. 

We next show that we can always satisfy the equation 
$x_{0}x_{1}x_{2}\cdots x_{n-1}-1=0$ with our solution. 
First, we perform an additional change of coordinates to separate 
the $\gamma$ coefficient.  
We let $y_{0}=Y_{0}$, $y_{1}=Y_{1}$, $\ldots$, 
$y_{m-2}=Y_{m-2}$, $y_{m-1}=\gamma_{}Y_{m-1}$.
Then on substitution of (\ref{gensol-rewritten}) 
into $x_{0}x_{1}x_{2}\cdots x_{n-1}-1=0$, we get

\begin{equation} \label{last-eq-reduction}
   \begin{array}{rcl}
     (\gamma^{m\ell}~\underbrace{u^{0} u^{0} \cdots u^{0}}_{m} 
\underbrace{u^{1} u^{1} \cdots  u^{1}}_{m} 
\cdots
\underbrace{u^{m\ell-1} u^{m\ell-1} \cdots u^{m\ell-1}}_{m} 
~Y_{0}^{m\ell}Y_{1}^{m\ell}Y_{2}^{m\ell} 
\cdots Y_{m-2}^{m\ell}Y_{m-1}^{m\ell})-1 & = & 0 \\
\\
\vspace{0.2in}
      (\gamma^{m\ell}~u^{m(0+1+2+\dots+(m\ell-1))}~
Y_{0}^{m\ell}Y_{1}^{m\ell}Y_{2}^{m\ell} 
\cdots Y_{m-2}^{m\ell}Y_{m-1}^{m\ell})-1 & = & 0 \\
\vspace{0.1in}
   (\gamma~u^{\frac{m(m\ell-1)}{2}}~Y_{0}Y_{1}Y_{2} 
\cdots Y_{m-2}Y_{m-1})^{m\ell}-1 & = & 0. \\
   \end{array}
\end{equation}
The last equation in (\ref{last-eq-reduction}) has now the same form 
as in \cite[Lemma~1.1]{Fau01}. 
We are done if we can satisfy it. 
We next show that it can always be satisfied with our solution. 

Since all the tropisms in the cone add up to zero, 
the product $(Y_{0}Y_{1}Y_{2} \cdots Y_{m-2}Y_{m-1})$, 
which consists of free parameter combinations, equals to 1. 
Since $(Y_{0}Y_{1}Y_{2} \cdots Y_{m-2}Y_{m-1})$ = 1, we are left with 

\begin{equation} \label{gamma-u}
   \begin{array}{rcl}
     (\gamma~u^{\frac{m(m\ell-1)}{2}})^{m\ell}-1=0.
   \end{array}
\end{equation} 
We distinguish two cases:
\begin{enumerate}
\item $\gamma=1$, implied by
     ($m$ is even, $\ell$ is odd) or ($m$ is odd, $\ell$ is odd)
      or ($m$ is even, $\ell$ is even). 

To show that (\ref{gamma-u}) is satisfied,
we rewrite (\ref{gamma-u}):

\begin{equation} \label{gamma-u-rewritten_case1}
     (u^{\frac{m(m\ell-1)}{2}})^{m\ell}-1 =  0 
\quad \Leftrightarrow  \quad (u^{\frac{m^2\ell(m\ell-1)}{2}})-1 = 0 
\quad \Leftrightarrow  \quad ((u^{m\ell})^{\frac{m(m\ell-1)}{2}})-1 =  0,
\end{equation} 
which is satisfied by $u = e^{\frac{i 2 \pi}{m\ell}}$
and $m (m \ell-1)$ being even.

\item  $\gamma \not=1$, implied by ($m$ is odd, $\ell$ is even).

To show that our solution satisfies (\ref{gamma-u}),
we rewrite (\ref{gamma-u}):

\begin{equation} \label{gamma-u-rewritten_case2_I}
     (\gamma~u^{\frac{m(m\ell-1)}{2}})^{m\ell}-1 = 0
\quad \Leftrightarrow \quad
     (\gamma~u^{\frac{m^2\ell}{2}} ~u^{\frac{-m}{2}})^{m\ell}-1 = 0
\quad \Leftrightarrow \quad
     (\gamma~  (u^{m\ell})^{\frac{m}{2}}~    u^{\frac{-m}{2}}    )^{m\ell}-1 
     =  0.
\end{equation} 

Since $u = e^{\frac{i 2 \pi}{m\ell}}$, $u^{m\ell}=1$, 
we can simplify (\ref{gamma-u-rewritten_case2_I}) further

\begin{equation} \label{gamma-u-rewritten_case2_II}
   \begin{array}{rcl}
     (\gamma~  u^{\frac{-m}{2}})^{m\ell}-1 & = & 0\\
     (e^{\frac{i\pi}{m\ell}}~ 
     (e^{\frac{i 2 \pi}{m\ell}})^{\frac{-m}{2}})^{m\ell}-1 & = & 0 \\
     (e^\frac{i\pi}{m\ell}~  (e^{\frac{-i \pi}{\ell}}))^{m\ell}-1 & = & 0 \\
     (e^{i\pi}~  e^{-i \pi m})-1 & = & 0 \\
     (e^{(1-m)i\pi})-1 & = & 0. \\
   \end{array}
\end{equation} 

Since $m$ is odd, we can write $m=2j+1$, for some~$j$.
The last equation of (\ref{gamma-u-rewritten_case2_II}) has the form 

\begin{equation} \label{gamma-u-rewritten_case2_III}
     (e^{(1-m)i\pi})-1 = 0 \quad \Leftrightarrow \quad
     (e^{(1-(2j+1))i\pi})-1 =  0 \quad \Leftrightarrow \quad
     (e^{(-2j)i\pi})-1 = 0.
\end{equation}
Since $(e^{(-2j)i\pi})$ = 1, for any $j$,
the equation $(e^{(-2j)i\pi})-1=0$ is satisfied,
implying (\ref{gamma-u}).
\end{enumerate}
\end{proof}

Backelin's Lemma comes to aid when applying a homotopy to find all
isolated cyclic $n$-roots as follows.  We must decide at the
end of a solution path whether we have reached an isolated solution or a
positive dimension solution set.  This problem is especially difficult
in the presence of isolated singular solutions 
(such as 4-fold isolated cyclic 9-roots~\cite{LVZ06}).
With the form of the solution set as in Backelin's Lemma, 
we solve a triangular binomial system in the parameters $t$ 
and with as $x$ values the solution found at the end of a path.
If we find values for the parameters for an end point,
then this solution lies on the solution set.

\section{Exploiting Symmetry}

We illustrate the exploitation of permutation symmetry on 
the cyclic 5-roots system.  Adjusting polyhedral homotopies
to exploit the permutation symmetry for this system was presented
in~\cite{VG95}.  

\subsection{The Cyclic 5-roots Problem}

The mixed volume for the cyclic 5-roots system is 70,
which equals the exact number of roots.
The first four equations of 
the cyclic 5-roots system $C_5(\x) = \zero$,
define solution curves:

\begin{equation} \label{c5sys4eq}
   \bff(\x) = 
   \left\{
      \begin{array}{l}
         x_{0} + x_{1} + x_{2} + x_{3} + x_{4} = 0 \\
         x_{0} x_{1} + x_{0} x_{4} + x_{1} x_{2} 
         + x_{2} x_{3} + x_{3} x_{4} = 0 \\ 
         x_{0} x_{1} x_{2} + x_{0} x_{1} x_{4} + x_{0} x_{3} x_{4} 
         + x_{1} x_{2} x_{3} + x_{2} x_{3} x_{4} = 0 \\
         x_{0} x_{1} x_{2} x_{3} + x_{0} x_{1} x_{2} x_{4}
         + x_{0} x_{1} x_{3} x_{4}
         + x_{0} x_{2} x_{3} x_{4} + x_{1} x_{2} x_{3} x_{4} = 0. \\
     \end{array}
   \right.
\end{equation}
where $\bfv = (1,1,1,1,1)$.  
As the first four equations of $C_5$ are homogeneous,
the first four equations of $C_5$ coincide with the first four
equations of ${\rm in}_\bfv(C_5)(\x) = \zero$.
Because these four equations are homogeneous, we have lines of solutions.
After computing representations for the solution lines,
we find the solutions to the original cyclic 5-roots problem
intersecting the solution lines with the hypersurface defined
by the last equation.  In this intersection, the exploitation
of the symmetry is straightforward.

The unimodular matrix with $\bfv = (1,1,1,1,1)$ 
and its corresponding coordinate transformation are
\begin{equation} \label{c5relxz}
   M =
     \begin{bmatrix}
      1 & 1 & 1 & 1 & 1\\
      0 & 1 & 0 & 0 & 0\\
      0 & 0 & 1 & 0 & 0\\
      0 & 0 & 0 & 1 & 0\\
      0 & 0 & 0 & 0 & 1\\
    \end{bmatrix}
   \quad
   \x = \z^M :
   \left\{
      \begin{array}{l}
         x_{0} = z_{0} \\
         x_{1} = z_{0} z_{1} \\
         x_{2} = z_{0} z_{2} \\
         x_{3} = z_{0} z_{3} \\
         x_{4} = z_{0} z_{4}. \\
      \end{array}
   \right.
\end{equation}
Applying $\x = \z^M$ to the initial form system (\ref{c5sys4eq}) gives
\begin{equation} \label{c5infz}
   {\rm in}_\bfv(\bff)(\x = \z^M) = 
   \left\{
      \begin{array}{l}
         z_{1} + z_{2} + z_{3} + z_{4} + 1 = 0 \\
         z_{1}z_{2} + z_{2}z_{3} + z_{3}z_{4} + z_{1} + z_{4} = 0 \\
         z_{1}z_{2}z_{3} + z_{2}z_{3}z_{4} + z_{1}z_{2} 
          + z_{1}z_{4} + z_{3}z_{4} = 0 \\
         z_{1}z_{2}z_{3}z_{4} + z_{1}z_{2}z_{3} + z_{1}z_{2}z_{4}
          + z_{1}z_{3}z_{4} + z_{2}z_{3}z_{4} = 0. \\
     \end{array}
   \right.
\end{equation}

The system (\ref{c5infz}) has 14 isolated solutions of the 
form $z_{1} = c_{1}$,~ $z_{2} = c_{2}$,~ $z_{3} = c_{3}$,~ $z_{4} = c_{4}$. 
If we let $z_0 = t$, in the original coordinates we have
\begin{equation} \label{c5sollines}
x_{0} = t,~
x_{1} = t c_{1},~
x_{2} = t c_{2},~
x_{3} = t c_{3},~
x_{4} = t c_{4}
\end{equation}
as representations for the 14 solution lines.

Substituting (\ref{c5sollines}) into the omitted 
equation $x_{0}x_{1}x_{2}x_{3}x_{4} - 1 = 0$, 
yields a univariate polynomial in $t$ of the form $kt^5-1 = 0$, 
where $k$ is a constant.
Among the 14 solutions, 10 are of the form $t^5 - 1$. 
They account for $10 \times 5 = 50$ solutions. There are 
two  solutions of the form $(-122.99186938124345) t^{5} - 1$, 
accounting for $2 \times 5 = 10$ solutions and an additional 
two solutions are of the form $(-0.0081306187557833118)t^{5} - 1$ 
accounting for $2 \times 5 = 10$ remaining solutions. 
The total number of solutions is 70, as indicated by the mixed volume 
computation.
Existence of additional symmetry, which can be exploited, 
can be seen in the relationship between the coefficients 
of the quintic polynomial, 
i.e. $\frac{1}{(-122.99186938124345)} \approx -0.0081306187557833118$. 

\subsection{A General Approach}

That the first $n-1$ equations of cyclic $n$-roots system give explicit
solution lines is exceptional.  For general polynomial systems
we can use the leading term of the Puiseux series 
to compute witness sets~\cite{SVW05}
for the space curves defined by the first $n-1$ equations.  
Then via the diagonal homotopy~\cite{SVW04}
we can intersect the space curves with the rest of the system.
While the direct exploitation of symmetry with witness sets is not possible,
with the Puiseux series we can pick out the generating space curves.

\section{Computing Pretropisms}

Following from the second theorem of Bernshte\v{\i}n~\cite{Ber75},
the Newton polytopes may be in general position and no normals
to at least one edge of every Newton polytope exists.  In that case,
there does not exist a positive dimensional solution set either.
We look for $\bfv$ so that ${\rm in}_\bfv(\bff)(\x) = \zero$
has solutions in~$(\cc^*)^n$ and therefore we look for pretropisms.
In this section we describe two approaches to compute pretropisms.
The first approach applies cddlib~\cite{FP96} on the Cayley embedding.
Algorithms to compute tropical varieties are described in~\cite{BJSST07}
and implemented in Gfan~\cite{Jen08}.  The second approach is the
application of {\tt tropical\_intersection} of Gfan.

\subsection{Using the Cayley Embedding}

The Cayley trick formulates a resultant as a discriminant
as in~\cite[Proposition 1.7, page 274]{GKZ94}.  
We follow the geometric description of~\cite{Stu93},
see also~\cite[\S9.2]{DRS10}.
The {\em Cayley embedding $E_{\bf A}$ 
of ${\bf A} = (A_1,A_2,\ldots,A_N)$} is
\begin{equation}
  E_{\bf A} = \left( A_1 \times \{ \zero \} \right)
   \cup \left( A_2 \times \{ \bfe_1 \} \right) \cup \cdots
   \cup \left( A_N \times \{ \bfe_{N-1} \} \right)
\end{equation}
where $\bfe_k$ is the $k$th $(N-1)$-dimensional unit vector.
Consider the convex hull of the Cayley embedding,
the so-called Cayley polytope, denoted by~${\rm conv}(E_{\bf A})$.  
If ${\rm dim}(E_{\bf A}) = k < 2n-1$,
then a facet of ${\rm conv}(E_{\bf A})$ is a face of dimension $k-1$.

\begin{proposition} Let $E_{\bf A}$ be the Cayley embedding of
the supports ${\bf A}$ of the system~$\bff(\x) = \zero$.
The normals of those facets of ${\rm conv}(E_{\bf A})$ 
that are spanned by at least two points of each support in~$\bf A$ 
form the tropical prevariety of~$\bff$.
\end{proposition}
\begin{proof}
Denote the Minkowski sum of the supports in $\bf A$ as
$\Sigma_{\bf A} = A_1 + A_2 + \cdots + A_N$.
Facets of~$\Sigma_{\bf A}$ spanned by at least two points of each support
define the generators of the cones of the tropical prevariety.
The relation between $E_{\bf A}$ and $\Sigma_{\bf A}$
is stated explicitly in~\cite[Observation~9.2.2]{DRS10}.
In particular, cells in a polyhedral subdivision of~$E_{\bf A}$
are in one-to-one correspondence with cells in a polyhedral subdivision
of the Minkowski sum $\Sigma_{\bf A}$.
The correspondence with cells in a polyhedral subdivision implies
that facet normals of~$\Sigma_{\bf A}$ occur as facet normals
of~${\rm conv}(E_{\bf A})$.
Thus the set of all facets of~${\rm conv}(E_{\bf A})$ gives
the tropical prevariety of~$\bff$.
\end{proof}

Note that $\Sigma_{\bf A}$ can be computed as the Newton polytope
of the product of all polynomials in~$\bff$.
As a practical matter, applying the Cayley embedding is better than
just plainly computing the convex hull of the Minkowski sum because
the Cayley embedding maintains the sparsity of the input,
at the expense of increasing the dimension.
Running {\tt cddlib} \cite{FP96} to compute the H-representation of the
Cayley polytope of the cyclic 8-roots problem yields 94 pretropisms.
With symmetry we have 11 generators, displayed in Table~\ref{table-cones}.
\begin{table}[hbt]
\begin{center}
\begin{tabular}{rcc|cccc}
\multicolumn{3}{c|}{generating pretropisms and initial forms}
& \multicolumn{4}{c}{~} \\ 
&                       & \#solutions of 
& \multicolumn{4}{c}{~higher dimensional cones of pretropisms} \\ 

& pretropism \textbf{v} & ${\rm in}_{\textbf{v}}(C_{8})(\textbf{z})$ &
1D & 2D & 3D & 4D \\ \hline
1. & $(-3, 1, 1, 1, -3,  1, 1, 1)$      & ~94 &
\{1\} & \{1, 3\} & \{1, 6, 11\} & \{1, 2, 3, 11\}  \\
2. & $(-1, -1, -1, 3, -1, -1, -1, 3)$   & 115 &
\{2\} & \{1, 6\} & \{1, 10, 11\} &   \\
3. & $(-1, -1, 1, 1, -1, -1, 1, 1)$     & 112 &
\{3\} & \{1, 10\} & \{2, 8, 11\} &   \\
4. & $(-1, 0, 0, 0, 1, -1, 1, 0)$       & ~30 &
\{4\} & \{1, 11\} &  &   \\
5. & $(-1, 0, 0, 0, 1, 0, -1, 1)$       & ~23 &
\{5\} & \{2, 3\} &  &   \\
6. & $(-1, 0, 0, 1, -1, 1, 0, 0)$       & ~32 &
\{6\} & \{2, 8\} &  &   \\
7. & $(-1, 0, 0, 1, 0, -1, 1, 0)$       & ~40 &
\{7\} & \{3, 7\} &  &   \\
8. & $(-1, 0, 0, 1, 0, 0, -1, 1)$       & ~16 &
\{8\} & \{2, 11\} &  &   \\
9. & $(-1, 0, 1, -1, 1, -1, 1, 0)$      & ~39 &
\{9\} & \{6, 11\} &  &   \\
10. & $(-1, 0, 1, 0, -1, 1, -1, 1)$     & ~23 &
\{10\} & \{8, 11\} &  &   \\
11. & $(-1, 1, -1, 1, -1, 1, -1, 1)$    & 509 &
\{11\} & \{10, 11\} &  &   \\
\end{tabular}
\caption{Eleven pretropism generators of the cyclic 8-root problem,
the number of solutions of the corresponding initial form systems,
and the multidimensional cones they generate, as computed by Gfan.}
\label{table-cones}
\end{center}
\end{table}

For the cyclic 9-roots problem, the computation of the facets of 
the Cayley polytope yield 276 pretropisms, with 17 generators:
$(-2$, $1$, $1$, $-2$, $1$, $1$, $-2$, $1$, $1)$,
$(-1$, $-1$, $2$, $-1$, $-1$, $2$, $-1$, $-1$, $2)$, 
$(-1$, $0$, $0$, $0$, $0$, $1$, $-1$, $1$, $0)$,
$(-1$, $0$, $0$, $0$, $0$, $1$, $0$, $-1$, $1)$, 
$(-1$, $0$, $0$, $0$, $1$, $-1$, $0$, $1$, $0)$,
$(-1$, $0$, $0$, $0$, $1$, $-1$, $1$, $0$, $0)$, 
$(-1$, $0$, $0$, $0$, $1$, $0$, $-1$, $0$, $1)$,
$(-1$, $0$, $0$, $0$, $1$, $0$, $-1$, $1$, $0)$,
$(-1$, $0$, $0$, $0$, $1$, $0$, $0$, $-1$, $1)$,
$(-1$, $0$, $0$, $1$, $-1$, $0$, $1$, $-1$, $1)$,
$(-1$, $0$, $0$, $1$, $-1$, $0$, $1$, $0$, $0)$,
$(-1$, $0$, $0$, $1$, $-1$, $1$, $-1$, $0$, $1)$,
$(-1$, $0$, $0$, $1$, $-1$, $1$, $-1$, $1$, $0)$,
$(-1$, $0$, $0$, $1$, $-1$, $1$, $0$, $-1$, $1)$,
$(-1$, $0$, $0$, $1$, $0$, $-1$, $1$, $-1$, $1)$,
$(-1$, $0$, $0$, $1$, $0$, $0$, $-1$, $0$, $1)$,
and $(-1$, $0$, $1$, $-1$, $1$, $-1$, $0$, $1$, $0)$.
To get the structure of the two dimensional cones, a second run
of the Cayley embedding is needed on the smaller initial form
systems defined by the pretropisms.

The computations for $n=8$ and $n=9$ finished in less than a second
on one core of a 3.07Ghz Linux computer with 4Gb RAM.  For the cyclic
12-roots problem, {\tt cddlib} needed about a week to compute the
907,923 facets normals of the Cayley polytope.
Although effective, the Cayley embedding becomes too inefficient
for larger problems.

\subsection{Using {\tt tropical\_intersection} of Gfan}

The solution set of the cyclic 8-roots polynomial system consists of 
space curves. 
Therefore, all tropisms cones were generated by a single tropism. 
The computation of the tropical prevariety however, did not lead only 
to single pretropisms but also to cones of pretropisms. 
The cyclic 8-roots cones of pretropisms 
and their dimension are listed in Table~\ref{table-cones}.   
Since the one dimensional rays of pretropisms yielded initial form 
systems with isolated solutions and since all higher dimensional
cones are spanned by those one dimensional rays, we can conclude 
that there are no higher dimensional algebraic sets, 
as any two dimensional surface degenerates to a curve 
if we consider only one tropism.

For the computation of the tropical prevariety, 
the Sage 5.7/Gfan function {\tt tropical\_intersection()} ran 
(with default settings without exploitation of symmetry)
on an AMD Phenom II X4 820 processor with 6 GB of RAM, 
running GNU/Linux, see Table~\ref{table-gfan-time}.
As the dimension $n$ increases so does the running time,
but the relative cost factors are bounded by~$n$.

\begin{table}[hbt]
\begin{center}
\begin{tabular}{r|r|r|r}
$n$~ &  seconds~~ & hms format~~~ & ~factor \\ \hline
 8~  &     16.37~ &            16 s~ & 1.0    \\
 9~  &     79.36~ &        1 m 19 s~ & 4.8    \\
10~  &    503.53~ &        8 m 23 s~ & 6.3    \\
11~  &   3898.49~ &   1 h ~4 m 58 s~ & 7.7    \\
12~  & ~37490.93~ & ~10 h 24 m 50 s~ & 9.6    \\
\end{tabular}
\caption{Time to compute the tropical prevarieties for cyclic $n$-roots
with Sage 5.7/Gfan and the relative cost factors: 
for $n=12$, it takes 9.6 times longer than for $n=11$. }
\label{table-gfan-time}
\end{center}
\end{table}

\section{The Second Term of a Puiseux Series}

In exceptional cases like the cyclic 4-roots problem 
where the first term of the series gives an exact solution 
or when we encounter solution lines like with the first
four equations of cyclic 5-roots, we do not have to look
for a second term of a series.
In general, a pretropism~$\bfv$ becomes a tropism 
if there is a Puiseux series with leading powers equal to~$\bfv$.
The leading coefficients of the series is a solution in~$\cc^*$ 
of the initial form system~${\rm in}_\bfv(\bff)(\x) = \zero$.
We solve the initial form systems with PHCpack~\cite{Ver99}
(its blackbox solver incorporates MixedVol~\cite{GLW05}).
For the computations of the series we use Sage~\cite{Sage}.

\subsection{Computing the Second Term}

In our approach, the calculation of the second term in the Puiseux
series is critical to decide whether a solution of an initial
form system corresponds to an isolated solution at infinity of
the original system, or whether it constitutes the beginning of
a space curve.  For sparse systems, we may not assume that the
second term of the series is linear in~$t$.
Trying consecutive powers of~$t$ will be wasteful for high degree
second terms of particular systems.  In this section we explain
our algorithm to compute the second term in the Puiseux series.

A unimodular coordinate transformation $\x = \z^M$ with $M$
having as first row the vector~$\bfv$ turns the initial form
system ${\rm in}_\bfv(\bff)(\x) = \zero$ into
${\rm in}_{{\bf e}_1}(\bff)(\z) = \zero$ 
where ${\bf e}_1 = (1,0,\ldots,0)$ equals the first standard
basis vector.  When $\bfv$ has negative components,
solutions of ${\rm in}_\bfv(\bff)(\x) = \zero$ that are at infinity
(in the ordinary sense of having components equal to~$\infty$)
are turned into solutions in~$(\cc^*)^n$ of
${\rm in}_{{\bf e}_1}(\bff)(\z) = \zero$.

The following proposition states the existence of the exponent of
the second term in the series.  After the proof of the proposition
we describe how to compute this second term.

\begin{proposition}
Let $\bfv$ denote the pretropism and $\x=\z^{M}$ denote the unimodular 
coordinate transformation, generated by~$\bfv$. 
Let ${\rm in}_\bfv(\bff)(\x=\z^{M})$ denote the transformed initial form 
system with regular isolated solutions, forming the isolated solutions 
at infinity of the transformed polynomial system $\bff(\x=\z^{M})$. 
If the substitution of the regular isolated solutions at infinity 
into the transformed polynomial system $\bff(\x=\z^{M})$ does not 
satisfy the system entirely, then the constant terms of $\bff(\x=\z^{M})$
have disappeared, leaving at least one monomial $c_{\ell}t^{w_{\ell}}$ 
for some $f_\ell$ in $\bff(\x=\z^{M})$ with minimal value $w_{\ell}$. 
The minimal exponent $w_{\ell}$ is the candidate for the exponent 
of the second term in the Puiseux series.
\end{proposition}

\begin{proof} 
Let $\z = (z_0,z_1,\ldots,z_{n-1})$ 
and $\bar{\z} = (z_1, z_2, \ldots, z_{n-1})$ 
denote variables after the unimodular transformation. 
Let $(z_{0} = t, z_{1} = r_{1}, \ldots, z_{n-1} = r_{n-1})$ 
be a regular solution at infinity and $t$ the free variable.

The $i$th equation of the original system after the unimodular 
coordinate transformation has the form
\begin{equation} \label{prop_sys}
   f_{i} = z_0^{m_{i}} ( P_{i}(\bar{\z}) + O(z_0) Q_{i}(\z) ),
   \quad i = 1,2, \ldots, N,
\end{equation}
where the polynomial $P_{i}(\bar{\z})$ consists of all monomials 
which form the initial form component of $f_{i}$ and $Q_{i}(\z)$ 
is a polynomial consisting of all remaining monomials of $f_{i}$. 
After the coordinate transformation, we denote the series expansion as
\begin{equation}
  \left\{
\begin{array}{l}
\displaystyle z_{0} = t \\
\displaystyle z_{j} = r_{j} + k_{j}t^{w_{\ell}}(1 + O(t)), 
\quad j = 1,2,\ldots,n-1.
\end{array} 
  \right.
\label{prop_sys_sol_form}
\end{equation}
for some $\ell$ and where at least one $k_j$ is nonzero.

We first show that, for all $i$, 
the polynomial $z_{0}^{m_{i}}P_{i}(\bar{\z})$ 
cannot contain a monomial of the form $c_{\ell}t^{w_{\ell}}$ on 
substitution of (\ref{prop_sys_sol_form}).
The polynomial $z_{0}^{m_{i}}P_{i}(\bar{\z})$ 
is the initial form of $f_{i}$, hence solution 
at infinity $(z_0=t, z_1=r_{1}, z_2=r_{2}, \ldots, z_{n-1}=r_{n-1})$ 
satisfies $z_{0}^{m_{i}}P_{i}(\bar{\z})$ entirely. 
Substituting (\ref{prop_sys_sol_form}) 
into $z_{0}^{m_{i}}P_{i}(\bar{\z})$ 
eliminates all constants in $t^{m_{i}}P_{i}(\bar{\z})$. 
Hence, the polynomial $P_{i}(t)$  = $R_{i}(t^{w})$ and, 
therefore, $t^{m_{i}}P_{i}(t)$  = $R_{i}(t^{w + m_{i}})$.

We next show that for some $i = \ell$, 
the polynomial $Q_{i}(\z)$ contains 
a monomial $c_{\ell}t^{w_{\ell}}$.
The polynomial $Q_{i}(\z)$ is rewritten:
\begin{equation}
   z_{0}^{w_{\ell}}Q_{i}(\bar{\z})
   = z_{0}^{w_{\ell}}T_{i0}(\bar{\z}) + z_{0}^{w_{\ell}+1}T_{i1}(\bar{\z})
   + \cdots.
\end{equation}
The polynomial $Q_{i}(\z)$ = $z_{0}^{w_{\ell}}Q_{i}(\bar{\z})$ 
consists of monomials which are not part of the initial form of~$f_{i}$. 
Hence, on substitution of solution at infinity (\ref{prop_sys_sol_form}), 
$z_{0}^{w_{\ell}} Q_{i}(\bar{\z})$ = $t^{w_{\ell}} Q_{i}(t)$  
does not vanish entirely and there must be at least one $i = \ell$
for which constants remain after substitution. 
Since $Q_{\ell}(t)$ contains monomials which are constants, 
$t^{w_{\ell}} Q_{\ell}(t)$ must contain a monomial of the 
form $c_{\ell} t^{w_{\ell}}$.
\end{proof}

Now we describe the computation of the second term,
in case the initial root does not satisfy the entire original system.
Assume the following general form of the series:
\begin{equation} \label{eqgenform}
   \left\{
      \begin{array}{l}
         z_0 = t \\
         z_i = c_i^{(0)} + k_i t^{w_{\ell}} (1 + O(t)),
         \quad i = 1,2,\ldots,n-1,
      \end{array}
   \right.
\end{equation}
for some~$\ell$ and
where $c_i^{(0)} \in \cc^*$ are the coordinates of the initial root,
$k_i$ is the unknown coefficient of the second term 
$t^{w_{\ell}}$, $w_{\ell} > 0$.
Note that only for some $k_i$ nonzero values may exist,
but not all $k_i$ may be zero.
We are looking for the smallest $w_{\ell}$ for which the linear system
in the $k_i$'s admits a solution with at least one nonzero coordinate.
Substituting~(\ref{eqgenform}) gives equations of the form
\begin{equation} \label{eqgenformsys}
   \widehat{c_i^{(0)}} t^{w_{\ell}} (1 + O(t))
   + t^{w_{\ell} + b_i} \sum_{j=1}^n \gamma_{ij} k_j (1 + O(t)) = 0,
   \quad i = 1,2, \ldots, n,
\end{equation}
for constant exponents $w_{\ell}$, $b_i$ and
constant coefficients $\widehat{c_i^{(0)}}$ and~$\gamma_{ij}$.

In the equations of~(\ref{eqgenformsys})
we truncate the $O(t)$ terms and retain those equations
with the smallest value of the exponents~$w_{\ell}$,
because with the second term of the series solution 
we want to eliminate the lowest powers of~$t$ when we
plug in the first two terms of the series into the system.
This gives a condition on the value~$w_{\ell}$ of the unknown exponent of~$t$
in the second term.  If there is no value for~$w_{\ell}$ so that we can
match with $w_{\ell} + b_i$ the minimal value of~$w_{\ell}$ for all equations
where the same minimal value of~$w_{\ell}$ occurs, then there does not
exist a second term and hence no space curve.
Otherwise, with the matching value for~$w_{\ell}$ we obtain a linear system
in the unknown $k$ variables.  If a solution to this linear system
exists with at least one nonzero coordinate, then we have found
a second term, otherwise, there is no space curve.

For an algebraic set of dimension~$d$, 
we have a polyhedral cone of $d$ tropisms and we take any general
vector $\bfv$ in this cone.  Then we apply the method outlined above
to compute the second term in the series in one parameter,
in the direction of~$\bfv$.

\subsection{Series Developments for Cyclic 8-roots}

We illustrate our approach on the cyclic 8-roots problem,
denoted by~$C_8(\x) = \zero$ and take as
pretropism $\bfv = (1, -1, 0, 1, 0, 0, -1, 0)$.
Replacing the first row of the 8-dimensional identity matrix by~$\bfv$ 
yields a unimodular coordinate transformation, denoted as $\x = \z^M$,
explicitly defined as
\begin{equation} \label{relxz}
x_{0} = z_{0},~
x_{1} = z_{1}/z_{0},~
 x_{2} = z_{2},~
 x_{3} = z_{0} z_{3},~
 x_{4} = z_{4},~
 x_{5} = z_{5},~
 x_{6} = z_{6}/z_{0},~
 x_{7} = z_{7}.
\end{equation}
Applying $\x = \z^M$ to the initial form system 
${\rm in}_\bfv(C_8)(\x) = \zero$ gives
\begin{equation} \label{infz1}
   {\rm in}_\bfv (C_{8})(\x = \z^M) = 
   \left\{
      \begin{array}{l}
          z_{1} + z_{6}=0\\ 
          z_{1} z_{2} + z_{5} z_{6} + z_{6} z_{7} = 0 \\ 
          z_{4} z_{5} z_{6} + z_{5} z_{6} z_{7} = 0 \\ 
          z_{4} z_{5} z_{6} z_{7} + z_{1} z_{6} z_{7} = 0 \\ 
          z_{1} z_{2} z_{6} z_{7} + z_{1} z_{5} z_{6} z_{7} = 0 \\
          z_{1} z_{2} z_{3} z_{4} z_{5} z_{6} 
 + z_{1} z_{2} z_{5} z_{6} z_{7} + z_{1} z_{4} z_{5} z_{6} z_{7} = 0 \\ 
          z_{1} z_{2} z_{3} z_{4} z_{5} z_{6} z_{7}
 + z_{1} z_{2} z_{4} z_{5} z_{6} z_{7} = 0 \\
          z_{1} z_{2} z_{3} z_{4} z_{5} z_{6} z_{7} - 1 = 0.
     \end{array}
   \right.
\end{equation}
By construction of~$M$, observe that all polynomials have the
same power of~$z_0$, so~$z_0$ can be factored out.
Removing~$z_0$ from the initial form system, we find a solution 
\begin{equation}
   z_{0} = t,~
   z_{1} = -I,~
   z_{2} = \frac{-1}{2} - \frac{I}{2},~
   z_{3} = -1,~ 
   z_{4} = 1+I,~ 
   z_{5} = \frac{1}{2} + \frac{I}{2},~ 
   z_{6} = I,~
   z_{7} = -1 - I
\end{equation}
where $I = \sqrt{-1}$.  This solution is a regular solution.
We set~$z_0 = t$, where $t$ is the variable for the Puiseux series.
In the computation of the second term, 
we assume the Puiseux series of the form at the left of~(\ref{cyc8form}).
We first transform the cyclic 8-roots system $C_8(\x) = \zero$ 
using the coordinate transformation given by (\ref{relxz}) 
and then substitute the assumed series form into this new system. 
Since the next term in the series is of the form $k_{j}t^{1}$, 
we collect all the coefficients of $t^{1}$ and solve the linear system 
of equations. 
The second term in the Puiseux series expansion for the cyclic 8-root system, 
has the form as at the right of~(\ref{cyc8form}).

\begin{equation} \label{cyc8form}
      \begin{cases}
         z_0 = t\\
         z_1 = -I +c_{1}t\\
         z_2 = \frac{-1}{2} - \frac{I}{2}+c_{2}t\\
         z_3 = -1+c_{3}t\\
         z_4 = 1+I+c_{4}t\\
         z_5 = \frac{1}{2} + \frac{I}{2}+c_{5}t\\
         z_6 = I+c_{6}t\\
         z_7 = (-1 - I)+c_{7}t\\
      \end{cases}
\quad \quad
      \begin{cases}
          z_0 = t\\
          z_1 = -I +(-1-I)t\\
          z_2 = \frac{-1}{2} - \frac{I}{2}+\frac{1}{2}t\\
          z_3 = -1\\
          z_4 = 1+I-t\\
          z_5 = \frac{1}{2} + \frac{I}{2}-\frac{1}{2}t\\
          z_6 = I+(1+I)t\\
          z_7 = (-1 - I)+t\\
      \end{cases}
\end{equation} 
Because of the regularity of the solution of the initial form system
and the second term of the Puiseux series, we have a symbolic-numeric
representation of a quadratic solution curve. 

If we place the same pretropism in another row in the unimodular matrix,
then we can develop the same curve starting at a different coordinate plane.
This move is useful if the solution curve would not be in general
position with respect to the first coordinate plane.
For symmetric polynomial systems, we apply the permutations to
the pretropism, the initial form systems, and its solutions to find
Puiseux series for different solution curves, related to the generating
pretropism by symmetry.  

Also for the pretropism $\bfv = (1, -1, 1, -1, 1, -1, 1, -1)$,
the coordinate transformation is given by the unimodular matrix 
$M$ equal to the identity matrix, except for its first row~$\bfv$.
The coordinate transformation $\x = \z^M$ 
yields $x_{0} = z_{0}$, $x_{1} = z_{1}/z_{0}$, $x_{2} = z_{0}z_{2}$, 
$x_{3} = z_{3}/z_{0}$, $x_{4} = z_{0}z_{4}$, $x_{5} = z_{5}/z_{0}$, 
$x_{6} = z_{0}z_{6}$, $x_{7} = z_{7}/z_{0}$. 
Applying the coordinate transformation 
to ${\rm in}_\bfv(C_8)(\x)$ gives
\begin{equation} \label{infz_onetermsol}
   {\rm in}_\bfv (C_{8})(\x = \z^M) = 
   \left\{
      \begin{array}{l}
          z_{1} + z_{3} + z_{5} + z_{7} = 0\\ 
          z_{1} z_{2} + z_{2} z_{3} + z_{3} z_{4} + z_{4} z_{5}
          + z_{5} z_{6} + z_{6} z_{7} + z_{1} + z_{7} = 0 \\ 
          z_{1} z_{2} z_{3} + z_{3} z_{4} z_{5} + z_{5} z_{6} z_{7} + z_{1} z_{7} = 0 \\ 
          z_{1} z_{2} z_{3} z_{4} + z_{2} z_{3} z_{4} z_{5} + z_{3} z_{4} z_{5} z_{6}
          + z_{4} z_{5} z_{6} z_{7} + z_{1} z_{2} z_{3} \\
          +~ z_{1} z_{2} z_{7} + z_{1} z_{6} z_{7} + z_{5} z_{6} z_{7} = 0 \\ 
          z_{1} z_{2} z_{3} z_{4} z_{5} + z_{3} z_{4} z_{5} z_{6} z_{7} 
          + z_{1} z_{2} z_{3} z_{7} + z_{1} z_{5} z_{6} z_{7} = 0 \\ 
          z_{1} z_{2} z_{3} z_{4} z_{5} z_{6} + z_{2} z_{3} z_{4} z_{5} z_{6} z_{7}
          + z_{1} z_{2} z_{3} z_{4} z_{5} + z_{1} z_{2} z_{3} z_{4} z_{7} \\
          +~ z_{1} z_{2} z_{3} z_{6} z_{7} + z_{1} z_{2} z_{5} z_{6} z_{7}
          + z_{1} z_{4} z_{5} z_{6} z_{7} + z_{3} z_{4} z_{5} z_{6} z_{7} = 0 \\ 
          z_{1} z_{2} z_{3} z_{4} z_{5} z_{6} z_{7}
          + z_{1} z_{2} z_{3} z_{4} z_{5} z_{7} 
          + z_{1} z_{2} z_{3} z_{5} z_{6} z_{7} 
          + z_{1} z_{3} z_{4} z_{5} z_{6} z_{7} = 0 \\ 
          z_{1} z_{2} z_{3} z_{4} z_{5} z_{6} z_{7} - 1 = 0
     \end{array}
   \right.
\end{equation}

The initial form system (\ref{infz_onetermsol}) has 72 solutions.
Among the 72 solutions, a solution of the form  
\begin{equation}
       z_{0} =t,~ 
       z_{1} =-1,~ 
       z_{2} =I,~ 
       z_{3} =-I,~ 
       z_{4} =-1,~ 
       z_{5} =1,~ 
       z_{6} =-I,~ 
       z_{7} =I, 
\end{equation}
here expressed in the original coordinates,
\begin{equation}
       x_{0} = t,~
       x_{1} = -1/t,~
       x_{2} = I t,~
       x_{3} = -I/t,~
       x_{4} = -t,~
       x_{5} = 1/t,~
       x_{6} = -I t,~
       x_{7} = I/t 
\end{equation}
satisfies the cyclic 8-roots entirely. Applying the cyclic permutation 
of this solution set we can obtain the remaining 7 solution sets, 
which also satisfy the cyclic 8-roots system.

In~\cite{Ver09b}, a formula for the degree of the curve
was derived, based on the coordinates of the tropism and
the number of initial roots for the same tropism.
We apply this formula 
and obtain 144 as the known degree of the space curve
of the one dimensional solution set, see Table~\ref{tabdegcyc8}.

\begin{table}[hbt]
\begin{center}
\begin{tabular}{cr}
   $(1, -1, 1, -1, 1, -1, 1, -1)$ & $8 \times 2 = 16$ \\
   $(1, -1, 0, 1, 0, 0, -1, 0) \to (1, 0, 0, -1, 0,1,-1,0)$
 & $8 \times 2 + 8 \times 2 = 32$ \\
   $(1, 0, -1, 0, 0, 1, 0, -1) \to (1, 0, -1,1,0,-1,0,0)$
 & $8 \times 2 + 8 \times 2 = 32$ \\
   $(1, 0, -1, 1, 0, -1, 0, 0) \to (1, 0, -1, 0, 0,1,0,-1)$
 & $8 \times 2 + 8 \times 2 = 32$ \\
   $(1, 0, 0, -1, 0, 1, -1, 0) \to (1, -1, 0,1,0,0,-1,0)$
 & $8 \times 2 + 8 \times 2 = 32$ \\
 & \multicolumn{1}{r}{TOTAL  = 144}
\end{tabular}
\caption{Tropisms, cyclic permutations, and degrees for the cyclic
         8 solution curve.}
\label{tabdegcyc8}
\end{center}
\end{table}

Using the same polyhedral method we can find all the isolated solutions 
of the cyclic 8-roots system.  We conclude this subsection with some
empirical observations on the time complexity.
In the direction $(1, -1, 0, 1, 0, 0, -1, 0)$, 
there is a second term in the Puiseux series as for the
40 solutions of the initial form system, there is no first term
that satisfies the entire cyclic 8-roots system. 
Continuing to construct the second term,
the total time required is 35.5 seconds, which 
includes 28 milliseconds that PHCpack needed to solve the initial form system. 
For $(1, -1, 1, -1, 1, -1, 1, -1)$ there is no second term 
in the Puiseux series as the first term satisfies the entire system.
Hence, the procedure for construction and computation
of the second term does not run. It takes PHCpack 12 seconds to solve the 
initial form system, whose solution set consists of 509 solutions. 
Determining that there is no second term for the 509 solutions, 
takes 199 seconds.
Given their numbers of solutions, the ratio for time comparison is given 
by $\frac{509}{40} = 12.725 $.  However, given that for tropisms 
$(1, -1, 0, 1, 0, 0, -1, 0)$ the procedure for construction and computation 
of the second term does run, 
unlike for tropism $(1, -1, 1, -1, 1, -1, 1, -1)$, 
the ratio for time comparison is not precise enough.
A more accurate ratio for comparison is $\frac{199}{35} \approx 5.686$.

\subsection{Cyclic 12-roots}

The generating solutions to the quadratic space curve solutions
of the cyclic 12-roots problem are in Table~\ref{c12_table}. 
As the result in the Table~\ref{c12_table} is given in the transformed 
coordinates, we return the solutions to the original coordinates. 
For any solution generator 
$(r_{1}, r_{2}, \ldots, r_{11})$ in Table~\ref{c12_table}: 

\begin{equation}
\begin{aligned}
      &z_{0} =t,~ z_{1} = r_{1},~ z_{2} = r_{2},~ z_{3} = r_{3},~ z_{4} = r_{4},~ z_{5} = r_{5}, \\
       &z_{6} =r_{6},~ z_{7} = r_{7},~ z_{8} = r_{8},~ z_{9} = r_{9},~ z_{10} = r_{10},~ z_{11} = r_{11}\\  
\end{aligned}
\end{equation}
and turning to the original coordinates we obtain
\begin{equation}
   \begin{aligned}
       &x_{0} =t,~ x_{1} =r_{1}/t,~ x_{2} = r_{2}t,~ x_{3} =r_{3}/t,~ x_{4} =r_{4}t,~ x_{5} =r_{5}/t\\ 
       &x_{6} =r_{6}t,~ x_{7} =r_{7}/t,~ x_{8} =r_{8}t,~ x_{9} =r_{9}/t,~ x_{10} =r_{10}t,~ x_{11} =r_{11}/t\\
   \end{aligned} 
\end{equation}
Application of the degree formula of~\cite{Ver09b} shows that
all space curves are quadrics.  
Compared to~\cite{Sab11}, we arrive at this result without
the application of any factorization methods.

\begin{sidewaystable}[hbt]
\resizebox{21cm}{!} {
\begin{tabular}{|c|c|c|c|c|c|c|c|c|c|c|c|}
  \hline
$z_1$ & $z_2$ & $z_3$ & $z_4$ & $z_5$ 
& $z_6$ & $z_7$ & $z_8$ & $z_9$ & $z_{10}$ & $z_{11}$ \\ \hline \hline
$\frac{1}{2}+ \frac{\sqrt{3}}{2}I$ & $\frac{1}{2}-\frac{\sqrt{3}}{2}I$ & $1$ & $-\frac{1}{2} -\frac{\sqrt{3}}{2}I$ & $\frac{1}{2} -\frac{\sqrt{3}}{2}I$ & $-1$ & $-\frac{1}{2} -\frac{\sqrt{3}}{2}I$ & $-\frac{1}{2}+ \frac{\sqrt{3}}{2}I$ & $-1$ & $\frac{1}{2}+ \frac{\sqrt{3}}{2}I$ & $-\frac{1}{2}+ \frac{\sqrt{3}}{2}I$ \\

$\frac{1}{2}+ \frac{\sqrt{3}}{2}I$ & $-1$ & $1$ & $1$ & $\frac{1}{2} -\frac{\sqrt{3}}{2}I$ & $-1$ & $-\frac{1}{2} -\frac{\sqrt{3}}{2}I$ & $1$ & $-1$ & $-1$ & $-\frac{1}{2}+ \frac{\sqrt{3}}{2}I$ \\

$\frac{1}{2}+ \frac{\sqrt{3}}{2}I$ & $\frac{1}{2}+ \frac{\sqrt{3}}{2}I$ & $-\frac{1}{2}-\frac{\sqrt{3}}{2}I$ & $1$ & $\frac{1}{2} -\frac{\sqrt{3}}{2}I$ & $-1$ & $-\frac{1}{2} -\frac{\sqrt{3}}{2}I$ & $-\frac{1}{2}+ \frac{\sqrt{3}}{2}I$ & $\frac{1}{2}+ \frac{\sqrt{3}}{2}I$ & $-1$ & $-\frac{1}{2}$ \\

$\frac{1}{2} -\frac{\sqrt{3}}{2}I$ & $-1$ & $-\frac{1}{2} -\frac{\sqrt{3}}{2}I$ & $1$ & $-1$ & $-1$ & $-\frac{1}{2}+ \frac{\sqrt{3}}{2}I$ & $1$ & $\frac{1}{2}+ \frac{\sqrt{3}}{2}I$ & $-1$ & $1$ \\

$-\frac{1}{2} -\frac{\sqrt{3}}{2}I$ & $\frac{1}{2} -\frac{\sqrt{3}}{2}I$ & $-1$ & $-\frac{1}{2}+ \frac{\sqrt{3}}{2}I$ & $-\frac{1}{2} -\frac{\sqrt{3}}{2}I$ & $-1$ & $\frac{1}{2}$ & $-\frac{1}{2}+ \frac{\sqrt{3}}{2}I$ & $1$ & $\frac{1}{2} -\frac{\sqrt{3}}{2}I$ & $\frac{1}{2}+ \frac{\sqrt{3}}{2}I$ \\

$1$ & $\frac{1}{2} -\frac{\sqrt{3}}{2}I$ & $-1$ & $1$ & $-\frac{1}{2} -\frac{\sqrt{3}}{2}I$ & $-1$ & $-1$ & $-\frac{1}{2}+ \frac{\sqrt{3}}{2}I$ & $1$ & $-1$ & $\frac{1}{2}+ \frac{\sqrt{3}}{2}I$ \\

$-\frac{1}{2}+ \frac{\sqrt{3}}{2}I$ & $\frac{1}{2} -\frac{\sqrt{3}}{2}I$ & $\frac{1}{2}+ \frac{\sqrt{3}}{2}I$ & $-\frac{1}{2} -\frac{\sqrt{3}}{2}I$ & $-\frac{1}{2} -\frac{\sqrt{3}}{2}I$ & $-1$ & $\frac{1}{2}+ \frac{\sqrt{3}}{2}I$ & $-\frac{1}{2}+ \frac{\sqrt{3}}{2}I$ & $-\frac{1}{2}+ \frac{\sqrt{3}}{2}I$ & $\frac{1}{2}+ \frac{\sqrt{3}}{2}I$ & $\frac{1}{2}+ \frac{\sqrt{3}}{2}I$ \\

$-\frac{1}{2} -\frac{\sqrt{3}}{2}I$ & $-1$ & $\frac{1}{2} -\frac{\sqrt{3}}{2}I$ & $1$ & $1$ & $-1$ & $\frac{1}{2}+ \frac{\sqrt{3}}{2}I$ & $1$ & $-\frac{1}{2}+ \frac{\sqrt{3}}{2}I$ & $-1$ & $-1$ \\

$1$ & $-1$ & $\frac{1}{2}+ \frac{\sqrt{3}}{2}I$ & $-\frac{1}{2} -\frac{\sqrt{3}}{2}I$ & $-\frac{1}{2} -\frac{\sqrt{3}}{2}I$ & $-1$ & $-1$ & $1$ & $-\frac{1}{2} -\frac{\sqrt{3}}{2}I$ & $\frac{1}{2}+ \frac{\sqrt{3}}{2}I$ & $\frac{1}{2}+ \frac{\sqrt{3}}{2}I$ \\

$-\frac{1}{2}+ \frac{\sqrt{3}}{2}I$ & $\frac{1}{2} -\frac{\sqrt{3}}{2}I$ & $\frac{1}{2} -\frac{\sqrt{3}}{2}I$ & $-\frac{1}{2} -\frac{\sqrt{3}}{2}I$ & $-\frac{1}{2}+ \frac{\sqrt{3}}{2}I$ & $-1$ & $\frac{1}{2}$ & $-\frac{1}{2}+ \frac{\sqrt{3}}{2}I$ & $-\frac{1}{2}+\frac{\sqrt{3}}{2}I$ & $\frac{1}{2}+ \frac{\sqrt{3}}{2}I$ & $\frac{1}{2} -\frac{\sqrt{3}}{2}I$ \\

$-\frac{1}{2}+ \frac{\sqrt{3}}{2}I$ & $\frac{1}{2} -\frac{\sqrt{3}}{2}I$ & $\frac{1}{2}+\frac{\sqrt{3}}{2}I$ & $-\frac{1}{2} -\frac{\sqrt{3}}{2}I$ & $1$ & $-1$ & $\frac{1}{2}$ & $-\frac{1}{2}+ \frac{\sqrt{3}}{2}I$ & $-\frac{1}{2} -\frac{\sqrt{3}}{2}I$ & $\frac{1}{2}+ \frac{\sqrt{3}}{2}I$ & $-1$ \\

$-\frac{1}{2}$ & $\frac{1}{2} -\frac{\sqrt{3}}{2}I$ & $\frac{1}{2} -\frac{\sqrt{3}}{2}I$ & $1$ & $1$ & $-1$ & $\frac{1}{2} -\frac{\sqrt{3}}{2}I$ & $-\frac{1}{2}+ \frac{\sqrt{3}}{2}I$ & $-\frac{1}{2}+ \frac{\sqrt{3}}{2}I$ & $-1$ & $-1$ \\

$\frac{1}{2} -\frac{\sqrt{3}}{2}I$ & $\frac{1}{2}-\frac{\sqrt{3}}{2}I$ & $-\frac{1}{2} -\frac{\sqrt{3}}{2}I$ & $-\frac{1}{2}+\frac{\sqrt{3}}{2}I$ & $\frac{1}{2} -\frac{\sqrt{3}}{2}I$ & $-1$ & $-\frac{1}{2}+ \frac{\sqrt{3}}{2}I$ & $-\frac{1}{2}+ \frac{\sqrt{3}}{2}I$ & $\frac{1}{2}+ \frac{\sqrt{3}}{2}I$ & $\frac{1}{2}+ \frac{\sqrt{3}}{2}I$ & $-\frac{1}{2}+ \frac{\sqrt{3}}{2}I$ \\

$\frac{1}{2}+ \frac{\sqrt{3}}{2}I$ & $\frac{1}{2} -\frac{\sqrt{3}}{2}I$ & $-\frac{1}{2} -\frac{\sqrt{3}}{2}I$ & $-\frac{1}{2} -\frac{\sqrt{3}}{2}I$ & $\frac{1}{2}+ \frac{\sqrt{3}}{2}I$ & $-1$ & $-\frac{1}{2} -\frac{\sqrt{3}}{2}I$ & $-\frac{1}{2} -\frac{\sqrt{3}}{2}I$ & $\frac{1}{2}+ \frac{\sqrt{3}}{2}I$ & $\frac{1}{2}+ \frac{\sqrt{3}}{2}I$ & $-\frac{1}{2} -\frac{\sqrt{3}}{2}I$\\

$-\frac{1}{2}+ \frac{\sqrt{3}}{2}I$ & $-1$ & $-1$ & $-\frac{1}{2} -\frac{\sqrt{3}}{2}I$ & $1$ & $-1$ & $\frac{1}{2} -\frac{\sqrt{3}}{2}I$ & $1$ & $1$ & $\frac{1}{2}+ \frac{\sqrt{3}}{2}I$ & $-1$ \\

$\frac{1}{2}+ \frac{\sqrt{3}}{2}I$ & $-1$ & $-\frac{1}{2}+ \frac{\sqrt{3}}{2}I$ & $-\frac{1}{2} -\frac{\sqrt{3}}{2}I$ & $\frac{1}{2} -\frac{\sqrt{3}}{2}I$ & $-1$ & $-\frac{1}{2} -\frac{\sqrt{3}}{2}I$ & $1$ & $\frac{1}{2} -\frac{\sqrt{3}}{2}I$ & $\frac{1}{2} -\frac{\sqrt{3}}{2}I$ & $-\frac{1}{2}+ \frac{\sqrt{3}}{2}I$ \\

$1$ & $\frac{1}{2} -\frac{\sqrt{3}}{2}I$ & $\frac{1}{2} -\frac{\sqrt{3}}{2}I$ & $-\frac{1}{2} -\frac{\sqrt{3}}{2}I$ & $-\frac{1}{2} -\frac{\sqrt{3}}{2}I$ & $-1$ & $-1$ & $-\frac{1}{2}+ \frac{\sqrt{3}}{2}I$ & $-\frac{1}{2}+ \frac{\sqrt{3}}{2}I$ & $\frac{1}{2}+ \frac{\sqrt{3}}{2}I$ & $\frac{1}{2}+ \frac{\sqrt{3}}{2}I$\\

$-1$ & $-1$ & $-\frac{1}{2} -\frac{\sqrt{3}}{2}I$ & $-\frac{1}{2} -\frac{\sqrt{3}}{2}I$ & $\frac{1}{2}+ \frac{\sqrt{3}}{2}I$ & $-1$ & $1$ & $1$ & $\frac{1}{2}$ & $\frac{1}{2}$ & $-\frac{1}{2} -\frac{\sqrt{3}}{2}I$ \\

$-1$ & $\frac{1}{2}$ & $1$ & $1$ & $\frac{1}{2}$ & $-1$ & $1$ & $-\frac{1}{2}+ \frac{\sqrt{3}}{2}I$ & $-1$ & $-1$ & $-\frac{1}{2} -\frac{\sqrt{3}}{2}I$ \\

$-\frac{1}{2} -\frac{\sqrt{3}}{2}I$ & $\frac{1}{2} -\frac{\sqrt{3}}{2}I$ & $\frac{1}{2} -\frac{\sqrt{3}}{2}I$ & $1$ & $-\frac{1}{2}$ & $-1$ & $\frac{1}{2}+ \frac{\sqrt{3}}{2}I$ & $-\frac{1}{2}+ \frac{\sqrt{3}}{2}I$ & $-\frac{1}{2} -\frac{\sqrt{3}}{2}I$ & $-1$ & $\frac{1}{2} -\frac{\sqrt{3}}{2}I$ \\

  \hline
\end{tabular}
}	
\caption{Generators of the roots of the 
initial form system ${\rm in}_{\bfv}(C_{12})(\x) = \zero$
with the tropism $\bfv = (+1,-1,+1,-1, +1,-1,+1,-1, +1,-1,+1,-1)$
in the transformed $\z$ coordinates.
Every solution defines a solution curve of 
the cyclic 12-roots system.}
\label{c12_table}
\end{sidewaystable}

\section{Concluding Remarks}

Inspired by an effective proof of the fundamental theorem
of tropical algebraic geometry, we outlined in this paper
a polyhedral method to compute Puiseux series expansions
for solution curves of polynomial systems.
The main advantage of the new approach is the capability
to exploit permutation symmetry.
For our experiments, we relied on
cddlib and Gfan for the pretropisms, the blackbox solver of PHCpack
for solving the initial form systems, and Sage for the manipulations
of the Puiseux series.  

\bibliographystyle{plain}

\end{document}